\documentclass[12pt]{amsart}

\usepackage{amssymb}
\usepackage{amsthm}
\usepackage{amsmath}
\usepackage{amsfonts}
\usepackage{amsfonts}
\usepackage{latexsym} 
\usepackage{eucal}
\usepackage{indentfirst} 

\newtheorem{theorem}{Theorem}[section]

\newtheorem{lemma}[theorem]{Lemma}

\newtheorem{remark}[theorem]{Remark}

\begin{document}

\def \cd {, \ldots ,}
\def \lf{\| f \|}
\def \ep{\varepsilon}

\def \pickmatrix{\left(\frac{1-w_j \overline w_k} {1-z_j \overline z_k}\right)_{j,k=1 \cd n}}


\title[Pick and Peak Interpolation]{Pick and Peak Interpolation}
\author{Alexander J. Izzo}
\address{Department of Mathematics and Statistics, Bowling Green State University, Bowling Green, OH 43403}
\email{aizzo@bgsu.edu}

\begin{abstract}
We show how Pick interpolation and interpolation on peak interpolation sets can be combined in an abstract uniform algebra setting.  In particular as a special case, the Rudin-Carleson theorem can be combined with the classical Pick interpolation theorem on the disc.  
\end{abstract}

\maketitle

\section{Introduction}
The classical Pick interpolation problem is this: given an $n$-tuple $(z_1 \cd z_n)$ of distinct points in the open unit disc $D$, for which $n$-tuples $(w_1 \cd w_n)$ of complex numbers does there exist a bounded holomorphic function $f$ on $D$ with supremum norm $\lf \le 1$ such that
$$f(z_j) = w_j,\ 1\le j \le n?   \leqno (*)$$
The answer is given by Pick's theorem \cite {Pi} which says that the desired function $f$ exists if and only if the matrix 
$$\pickmatrix$$
is positive semi-definite.  The proof of 
Pick's theorem can be found in many places, for instance \cite {Gar}.  When the Pick interpolation problem has a solution, the solution can, in fact, be taken to be a finite Blaschke product.  Thus there is no change in Pick's problem if in the statement of the problem we require the function to be in the disc algebra $A(D)$.  (The disc algebra is the algebra of all continuous complex-valued functions on the closed disc $\overline D$ that are holomorphic on $D$.)  However, it then becomes natural to consider a slightly more general situation in which some of the points $z_1 \cd z_n$ are allowed to lie on the boundary of $D$.  Actually we will consider something even more general; we will ask for interpolation at a finite number of points in the open disc together with interpolation on a peak interpolaton set on the boundary.  However, to have a tractable problem, a slight reformulation is needed.  We will replace the condition that $\lf \le 1$ by the requirement that for every $\ep > 0$ there is a solution $f$ to the interpolation problem with $\lf \le 1 + \ep$.  Note that this reformulation does not change the classical Pick problem;
if for every $\ep > 0$ there is a bounded holomophic function $f$ with 
$\lf \le 1 + \ep$ satisfying $\displaystyle(*)$, then a normal families argument shows that there is a bounded holomorphic function $f$ with $\lf \le 1$ satisfying $\displaystyle(*)$.

Generalization of Pick's problem to an abstract uniform algebra context has been extensively studied by Brian Cole, Keith Lewis, and John Wermer \cite {CoLeWe, CoWe, We}, 
and in that context they make use of the reformulation involving 
$1 +\ep$ just mentioned.  It arises naturally from considering a quotient norm.  Let $A$ be a uniform algebra on a compact space $X$, that is, an algebra of 
complex-valued continuous functions on $X$ that contains the constant functions, separates the points of $X$, and is uniformly closed in the algebra $C(X)$ of all complex-valued continuous functions on $X$.
Fix $n$ points $z_1 \cd z_n$ in $X$.  Define
$$I = \{ f \in A : f (z_j) = 0, 1 \le j \le n \}.$$
Then $I$ is a closed ideal in $A$.  The quotient algebra $A/I$ of $A$ by the ideal $I$ consists of all cosets $[f]$ relative to $I$ where $f \in A$.  The norm of $[f]$ in $A/I$ is given by 
$$\| [f] \| = \inf \{ \| g \| : g \in [f] \}.$$
Let $w_1 \cd w_n$, be fixed complex numbers, and let $f$ be a function in $A$ such that $f (z_j) = w_j, 1 \le j \le n$.
Then the condition that $\| [f] \| \le 1$ is exactly the 
condition that for every $\ep > 0$ there is a function 
$f_\ep \in A$  such that $f_\ep (z_j) = w_j, 1 \le j \le n$, with $\| f_\ep \| \le 1 + \ep$.

The proof of our theorem that Pick interpolation can be combined with interpolation on a peak interpolation set is no more difficult in an abstract uniform algebra context.  We will state and prove the theorem in that context and then show how it specializes to the disc algebra.  

This paper grew out of conversations that the author had with John Wermer.  The author thanks Wermer for these valuable conversations.

\section{Pick and peak interpolation}

Given a uniform algebra $A$ on a compact space $X$ and a subset $E$ of $X$, a function $g$ in $A$ is said to {\it peak\/} on $E$ if 
$g= 1$ on $E$ and 
$|g| < 1$ on $X \setminus E$.  
The subset $E$ of $X$ is said to be a {\it peak interpolation set\/} for $A$ if for every nonzero function $f$ in $C(E)$ there is a function $F$ in $A$ such that $F|E=f$ and $|F(x)|< \|f\|$ for every $x$ in $X\setminus E$.  Our main theorem is as follows. 
 
\begin{theorem}\label{Proposition2.1}
Suppose $A$ is a uniform algebra on a compact space $X$.   Fix a peak interpolation set $E$ for $A$ and points 
$\alpha_1 \cd \alpha_n$ in $X \setminus E$.  Also fix a continuous complex-valued function $f$ on $E$ with $\lf \le 1$ and complex numbers $w_1 \cd w_n$.  If for each $\ep > 0$ there exists a function $F$ in $A$ with $\| F \| \le 1 + \ep$ such that 
$F (\alpha_j) = w_j, 1 \le j \le n$, then the function $F$ can be chosen so that in addition $F|E=f$.
\end{theorem}

Specializing to the disc algebra gives the following corollary that combines Pick interpolation with the Rudin-Carleson theorem \cite {Carl, Rudin} (or see \cite [Theorem~II.12.6] {Gamelin}).  The proof is an immediate application of the above proposition 
and the theorems of Pick and of Rudin and Carleson.

\begin{theorem}\label{Corollary 2.2}
Let $E$ be a closed subset of $\partial D$ of 1-dimensional Lebesgue measure zero, let $f$ be a continuous complex-valued function on $E$ with $\lf \le 1$, let $(z_1 \cd z_n)$ be an n-tuple of distinct points in the open unit disc $D$, and let 
$(w_1 \cd w_n)$ be an n-tuple of complex numbers.  Then for every $\ep > 0$ there is a function $F$ in the disc algebra with $\| F \|
 \le 1 + \ep$ such that $F|E=f$ and $F(z_j)=w_j, 1 \le j \le n$, if and only if the matrix 
$$\pickmatrix$$
is positive semi-definite.
\end{theorem}

Before proving Proposition 2.1, we establish a lemma.

\begin{lemma}\label{Lemma:}
Given a uniform algebra $A$ on a compact space $X$ and a peak set $E$ for $A$, there is a sequence of functions $(f_m)_{m=1}^\infty$ in $A$ with $\| f_m \| \le 1$ for all $m$ such that $f_m = 0$ on $E$ for all $m$ but $f_m \rightarrow 1$ pointwise on $X \setminus E$.
\end{lemma}

\begin{proof}
Let $f$ be a function in $A$ that peaks on $E$.  Set $g = (1-f)/2$.  Then the range of $g$ lies in the closed disc $\Delta = \{ z: |z- {1 \over 2} | \le {1 \over 2} \}$.  Note that on $\Delta$ there is a well-defined continuous $m^{\rm th}$ root function $\gamma$ with $\gamma (1) =1$.  Since $\gamma$ is holomorphic on $\rm{Int} (\Delta)$, there is a sequence of polynomials $(p_k)_{k=1}^\infty$ such that $p_k \rightarrow \gamma$ uniformly on $\Delta$.  Consequently, $g$ has an $m^{\rm th}$ root $g^{1/m} = \gamma \circ g$ in $A$.  Note that $g^{1/m}=0$ on $E$ and $\| g^{1/m} \| \le 1$ (since $\| g \| \le 1$).  For $x \in X \setminus E$, we have $g(x) \ne 0$, so $g^{1/m} (x) \rightarrow 1$ as $m \rightarrow \infty$.  Thus setting $f_m = g^{1/m}$ gives the lemma.
\end{proof}

\begin{proof}[Proof of Theorem 2.1]
{\it Step 1}:
We show that for each $\ep >0$ there is a function $G$ in $A$ with $\| G \| \le 1 + \ep$ such that $G|E=f$ and $|G(\alpha_j) - w_j | < \ep\ (j=1 \cd n)$.

By hypothesis there is a function $l$ in $A$ with $\| l \| \le 1 +(\ep/2)$ and 
$l (\alpha_j)=w_j\ (j=1 \cd n)$.  Let $(f_m)$ be a sequence as in the lemma above.  Put
$$g_m (z)=f_m(z) l(z).$$
As $m \rightarrow \infty$, note that $g_m \rightarrow l$ pointwise on $X \setminus E$, so
$g_m(\alpha_j) \rightarrow w_j$ for each $j$.  Also
$ \| g \| \le \| l \| \le 1 + (\ep /2)$. Choose $m_0$ large enough that
$| g_{m_0} (\alpha_j) -w_j | < \ep /2$ for all $j$, and set $g=g_{m_0}$.  Since each $f_m$ vanishes on $E$, so does $g$, so we can choose a neighborhood $U$ of $E$ on which $|g| < \ep /2$.  Furthermore we can choose $U$ so as to contain none of the $\alpha_j$.

Since $E$ is a peak interpolation set, we can choose a function $h$ in $A$ with 
$\| h \| \le 1$ such that $h|E=f$ and $|h|< \ep /2$ on $X \setminus U$.  (To obtain such a function $h$, start with a function in $A$ satisfying the first two conditions, and then multiply by a sufficiently high power of a function that peaks on $E$ to achieve the last condition.)
Set
$$G=g+h.$$
Then $G$ is in $A$.
Applying the inequality $|G(x)|\le |g(x)|+|h(x)|$ separately for $x$ in $U$ and in $X \setminus U$, we see that $\| G \| \le 1 + \ep$.  Since $g=0$ on $E$, we have $G=h=f$ on $E$.  Finally 
$|G(\alpha_j)-w_j|=|g(\alpha_j)+h(\alpha_j)-w_j| \le |g(\alpha_j)-w_j|+|h(\alpha_j)|<\ep$.

{\it Step 2}:
We complete the proof.  Choose a function $r$ in $A$ that peaks on $E$, i.e., such that $r=1$ on $E$ and $|r|<1$ on $X \setminus E$.  Since the functions in $A$ separate points, there is a function $k_1$ in $A$ that vanishes at $\alpha_2 \cd \alpha_n$ but is nonzero at $\alpha_1$.  
Then the function $(1-r) k_1$ is zero on $E$ and at $\alpha_2 \cd \alpha_n$ but is nonzero at $\alpha_1$
By multiplying by a constant, we obtain a function $\varphi_1$ in $A$ that is $1$ at $\alpha_1$ and  zero on $E$ and at $\alpha_2 \cd \alpha_n$.  
Similarly we obtain, for each $j$, a function $\varphi_j$ in $A$ that is $1$ at 
$\alpha_j$ and zero on $E$ and at the other $\alpha$'s.  
Let $M= \max\limits_{1\le j\le n} \| \varphi_j \|$.
Take $G$ as in Step 1 and set $\sigma_j=G(\alpha_j)-w_j$.  Finally set
$$F=G-\sum^n_{j=1} \sigma_j \varphi_j.$$
Then $F$ is in $A$,
\begin{eqnarray*}
\| F \| &\le& \|G \| + \sum\limits^n_{j=1}  |\sigma_j| \| \varphi_j \| \\
&\le& \|G \|  + \ep n M \\
&\le& 1+ \ep (1+nM),
\end{eqnarray*}
$F|E= G|E=f$, and $F(\alpha_j)=G(\alpha_j) -\sigma_j=w_j$
for each $j$.  Since $\ep$ is arbitrary and $M$ is independent of $\ep$, this proves the proposition.
\end{proof}

\begin{remark}
{\rm Theorem 2.1 remains valid with \lq\lq peak interpolation set" replaced by \lq\lq generalized peak interpolation set" where by \lq\lq generalized peak interpolation set" we mean a set that is a generalized peak set (i.e., an intersection of peak sets) and an interpolation set.  The proof goes through almost unchanged.  We choose a peak set $E'$ containing $E$ and avoiding all the $\alpha_j$'s and apply the lemma to $E'$ to get a sequence $(f_m)$.  We obtain $g$ from $(f_m)$ as before and take $U$ to be a neighborhood of $E'$ avoiding the $\alpha_j$'s with $|g|<\ep$ on $U$.  We then obtain a function $h$ in $A$ with $\| h \| \le 1$ such that $h|E=f$ and $|h|< \ep$ on $X\setminus U$ as before.  The rest of Step 1 goes through as before.  Step 2 is unchanged except that we take $r$ to peak on $E'$.}
\end{remark}



\end{document}